\documentclass[11pt]{aart}

\usepackage[letterpaper, hmargin=1in, top=1in, bottom=1in, footskip=0.6in]{geometry}

\usepackage{titlesec}
\titleformat{\section}[block]{\filcenter\normalfont\bfseries\large}{\thesection.}{.5em}{}\titlespacing*{\section}{0pt}{2\baselineskip}{1\baselineskip}
\titleformat{\subsection}[runin]{\normalfont\bfseries}{\thesubsection.}{.4em}{}[.]\titlespacing{\subsection}{0pt}{2ex plus .1ex minus .2ex}{.8em}
\titleformat{\subsubsection}[runin]{\normalfont\itshape}{\thesubsubsection.}{.3em}{}[.]\titlespacing{\subsubsection}{0pt}{1ex plus .1ex minus .2ex}{.5em}
\titleformat{\paragraph}[runin]{\normalfont\itshape}{\theparagraph.}{.3em}{}[.]\titlespacing{\paragraph}{0pt}{1ex plus .1ex minus .2ex}{.5em}

\usepackage[T1]{fontenc}
\usepackage[utf8]{inputenc}
\usepackage{lmodern}

\usepackage[labelfont=bf,font=small,labelsep=period]{caption}
\usepackage{microtype}

\let\originalleft\left
\let\originalright\right
\renewcommand{\left}{\mathopen{}\mathclose\bgroup\originalleft}
\renewcommand{\right}{\aftergroup\egroup\originalright}

\usepackage{amsmath}
\usepackage{amssymb}
\usepackage{amsfonts}
\usepackage{latexsym}
\usepackage{amsthm}
\usepackage{amsxtra}
\usepackage{amscd}
\usepackage{bbm}
\usepackage{mathrsfs}
\usepackage{bm}
\usepackage{mathtools}

\usepackage{graphicx, color}

\definecolor{darkred}{rgb}{0.9,0,0.3}
\definecolor{darkblue}{rgb}{0,0.3,0.9}
\definecolor{purple}{rgb}{0.6,0,0.7}

\definecolor{vdarkred}{rgb}{0.6,0,0.2}
\definecolor{vdarkblue}{rgb}{0,0.2,0.6}
\usepackage[pdftex, colorlinks, linkcolor=vdarkblue,citecolor=vdarkred]{hyperref}

\usepackage{booktabs}
\usepackage[nottoc,notlof,notlot]{tocbibind}
\usepackage{cite} 

\numberwithin{equation}{section}

\usepackage{enumitem}

\theoremstyle{plain} 
\newtheorem{theorem}{Theorem}[section]
\newtheorem*{theorem*}{Theorem}
\newtheorem{lemma}[theorem]{Lemma}
\newtheorem*{lemma*}{Lemma}

\newtheorem*{corollary*}{Corollary}
\newtheorem{proposition}[theorem]{Proposition}
\newtheorem*{proposition*}{Proposition}

\newtheorem*{conjecture*}{Conjecture}

\theoremstyle{definition} 

\newtheorem*{definition*}{Definition}

\newtheorem*{example*}{Example}
\newtheorem{remark}[theorem]{Remark}
\newtheorem*{remark*}{Remark}

\newtheorem*{assumption*}{Assumption}

\newtheorem*{convention*}{Convention}

\newcommand{\f}{\mathbf} 
\newcommand{\bb}{\mathbb} 
\renewcommand{\cal}{\mathcal}

\newcommand{\wt}{\widetilde}

\renewcommand{\P}{\mathbb{P}}

\newcommand{\R}{\mathbb{R}}
\newcommand{\C}{\mathbb{C}}
\newcommand{\N}{\mathbb{N}}

\newcommand{\ee}{\mathrm{e}}
\newcommand{\ii}{\mathrm{i}}
\newcommand{\dd}{\mathrm{d}}
\newcommand{\col}{\mathrel{\vcenter{\baselineskip0.75ex \lineskiplimit0pt \hbox{.}\hbox{.}}}}
\newcommand*{\deq}{\mathrel{\vcenter{\baselineskip0.5ex \lineskiplimit0pt\hbox{\scriptsize.}\hbox{\scriptsize.}}}=}

\renewcommand{\leq}{\leqslant}
\renewcommand{\geq}{\geqslant}
\renewcommand{\epsilon}{\varepsilon}

\newcommand{\floor}[1] {\lfloor #1 \rfloor}

\newcommand{\ind}[1]{\mathbbm 1_{#1}}

\newcommand{\pb}[1]{\bigl(#1\bigr)}

\newcommand{\pbb}[1]{\biggl(#1\biggr)}

\newcommand{\qB}[1]{\Bigl[#1\Bigr]}
\newcommand{\qbb}[1]{\biggl[#1\biggr]}

\newcommand{\abs}[1]{\lvert #1 \rvert}

\newcommand{\absbb}[1]{\biggl\lvert #1 \biggr\rvert}

\newcommand{\norm}[1]{\lVert #1 \rVert}

\newcommand{\eps}{\varepsilon}

\renewcommand{\Im}{\mathrm{Im}\,} 					
\renewcommand{\Re}{\mathrm{Re}\,} 					

\title{The completely delocalized region of the Erd{\H o}s-R\'enyi graph} 
\author{Johannes Alt \and Raphael Ducatez \and Antti Knowles}

\begin{document}

\maketitle

\begin{abstract}
We analyse the eigenvectors of the adjacency matrix of the Erd{\H o}s-R\'enyi graph on $N$ vertices with edge probability $\frac{d}{N}$. We determine the full region of delocalization by determining the critical values of $\frac{d}{\log N}$ down to which delocalization persists: for $\frac{d}{\log N} > \frac{1}{\log 4 - 1}$ all eigenvectors are completely delocalized, and for $\frac{d}{\log N} > 1$ all eigenvectors with eigenvalues away from the spectral edges are completely delocalized. Below these critical values, it is known \cite{ADK20, ADK21} that localized eigenvectors exist in the corresponding spectral regions.
\end{abstract}

\section{Introduction} 

Let $A$ be the adjacency matrix of the Erd\H{o}s-R\'enyi graph $\mathbb{G}(N,d/N)$, defined as the random graph on $N$ vertices where each edge of the complete graph is kept with probability $d/N$ independently of the others. The subject of this note is the delocalization of the eigenvectors of $A$ in the limit of large $N$. A commonly used measure of delocalization of a vector $\f u \in \C^N$, which we also adopt here, is the quotient
\begin{equation*}
q(\f u) \deq \frac{\norm{\f u}^2_\infty}{\norm{\f u}_2^2},
\end{equation*}
where $\norm{\f u}_p$ denotes the $\ell^p$-norm of $\f u$.
Informally, $q(\f u) \asymp 1$ corresponds to a localized vector and $q(\f u) = N^{-1 + o(1)}$ to a completely delocalized vector. We refer to \cite{ADK20} and the references therein for an extensive discussion on the question of localization versus delocalization of the eigenvectors of random graphs and random matrices in general.

Complete delocalization for all eigenvectors of $A$ was established in \cite{EKYY1} for $d \geq (\log N)^6$ and in \cite{HeKnowlesMarcozzi2018} for $d \geq C \log N$ for some large constant $C$. The scale $d \asymp \log N$ is well known to be critical for the graph $\mathbb{G}(N,d/N)$, in the sense that it is the scale at which the concentration of the degrees fails. Above this scale, $\mathbb{G}(N,d/N)$ is typically homogeneous; below this scale, $\mathbb{G}(N,d/N)$ is typically inhomogeneous and exhibits isolated vertices, hubs, and leaves. Such structures may cause the delocalization of eigenvectors to fail. The most basic instance of this failure is the well-known connectivity threshold of $\mathbb{G}(N,d/N)$ at $d = \log N$ \cite{MR125031,Bol01}: 
for any constant $\kappa > 0$, if $d \geq (1 + \kappa) \log N$ then the graph is connected with high probability, while if $d \leq (1 - \kappa) \log N$ the graph contains isolated vertices with high probability. In the latter case, we trivially have localized eigenvectors at the origin. A much more subtle localization phenomenon, associated with hubs, was uncovered in \cite{ADK20} (see also \cite{ADK19, tikhomirov2021outliers, BBK1, BBK2} for previous results on eigenvalues): defining
\begin{equation} \label{b_star}
b_* \deq \frac{1}{\log 4 - 1} \approx 2.59,
\end{equation}
if $d \leq (b_* - \kappa) \log N$ then a \emph{semilocalized} phase exists near the edge of the spectrum, characterized by  $q(\f u) \geq N^{-\gamma}$ for some $\gamma < 1$. Moreover, in \cite{ADK21} it was proved that if $d \leq (b_* - \kappa) \log N$ then the extreme\footnote{Here, and throughout this introduction, we leave aside the largest eigenvalue of $A$, which for $d \asymp \log N$ is the Perron-Frobenius eigenvalue and constitutes an outlier separated from the rest of the spectrum.} eigenvectors of $A$ are localized.

In \cite{ADK20} it was proved that on the scale $d \asymp \log N$ complete delocalization persists in the spectral region $[-2 + \kappa, -\kappa] \cup [\kappa, 2 -\kappa]$ for the matrix $A / \sqrt{d}$. This spectral region excludes precisely the two regions exhibiting localized or semilocalized eigenvectors described in the previous paragraph: the neighbourhoods  $(-\kappa,\kappa)$ and $\pm (2 - \kappa, \infty)$ of the origin and of the spectral edges, respectively. In fact, in \cite{ADK20} it was proved that complete delocalization in this spectral region persists down to scales $d \gg \sqrt{\log N}$, below which it fails throughout the spectrum.

Hence, the question of when delocalization occurs in the neighbourhoods $(-\kappa,\kappa)$ and $\pm (2 - \kappa, \infty)$ was left open. We settle it here. Writing $d = b \log N$ for some constant $b>0$, we show that complete delocalization for $A /\sqrt{d}$ holds throughout the spectrum provided that $b > b_*$ and in $[-2 + \kappa, 2 - \kappa]$ provided that $b > 1$. As explained above, this result is optimal since for any $b < b_*$ there are localized states near the spectral edges, and for any $b < 1$ there are localized states in $(-\kappa,\kappa)$. Hence, combined with \cite{ADK20}, our result gives a complete description of the delocalized spectral region of $A / \sqrt{d}$. In addition, it shows that the extreme eigenvectors of $A / \sqrt{d}$ undergo a sharp transition from completely delocalized to localized as $b$ crosses $b_*$. We refer to Figure \ref{fig:phases} below for a phase diagram summarizing our results.

\begin{figure}[!ht]
\begin{center}
{\small 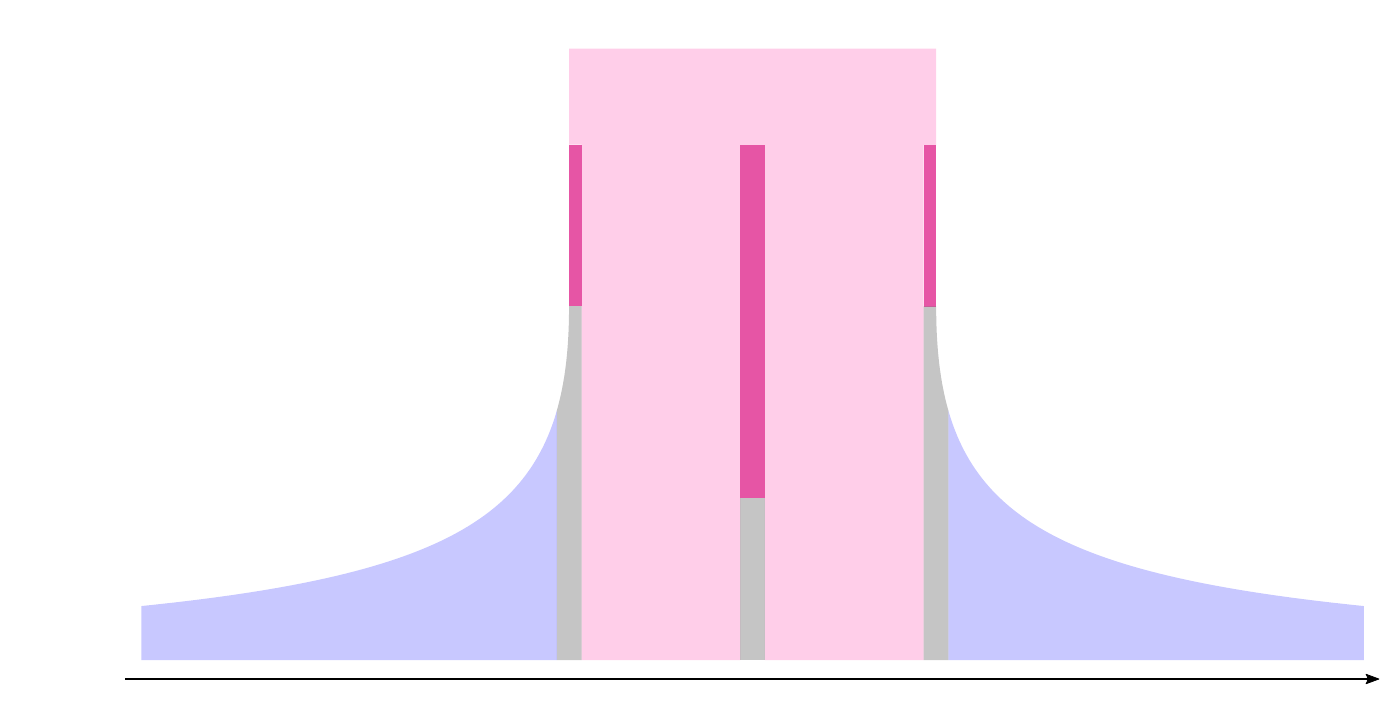}
\end{center}
\caption{
The phase diagram of the rescaled adjacency matrix $A / \sqrt{d}$ of the Erd\H{o}s-R\'enyi graph $\bb G(N,d/N)$ at criticality, where $d = b \log N$ with $b$ fixed. The horizontal axis records the location in the spectrum and the vertical axis the sparseness parameter $b$. Here, $b_*$ is defined in \eqref{b_star} and $C$ is the large constant from \cite{HeKnowlesMarcozzi2018}. The spectrum is confined to the coloured region \cite{ADK19, tikhomirov2021outliers}. This diagram summarises the results of \cite{ADK20, ADK21, HeKnowlesMarcozzi2018} and the present paper. In the red region the eigenvectors are completely delocalized; the light red region was established in \cite{ADK20, HeKnowlesMarcozzi2018} and the dark red region is established in the present paper. In the light blue region the eigenvectors are semilocalized \cite{ADK20}, and near the spectral edge (dark blue line) they are localized \cite{ADK21}. The grey regions have width $o(1)$ and have not been fully analysed yet.
\label{fig:phases}}
\end{figure}

We now state our main delocalization result.

\begin{theorem} \label{thm:delocalization_ER} 
For any constant $\kappa > 0$ the following holds with probability $1 - o(1)$.
\begin{enumerate}[label=(\roman*)] 
\item \label{item:deloc_everywhere} If $d \geq (b_* + \kappa) \log N$ then all eigenvectors $\f u$ of $A$ satisfy $q(\f u) \leq N^{-1 + \kappa}$.
\item \label{item:deloc_throughout_bulk} If $d \geq (1 + \kappa) \log N$ then all eigenvectors $\f u$ of $A$ whose eigenvalues are bounded in absolute value by $(2-\kappa) \sqrt{d}$ satisfy $q(\f u) \leq N^{-1 + \kappa}$.
\end{enumerate} 
\end{theorem}

Delocalization in the sense of $q(\f u) \leq N^{-1 + \kappa}$ has been a central topic in random matrix theory ever since the seminal work \cite{ESY2}.
The proof of Theorem \ref{thm:delocalization_ER} is based on an extension of the delocalization argument from \cite{ADK20}. There, it was shown that the spectral measure of the Green function of $A / \sqrt{d}$ at some vertex $x$ is well approximated in a certain spectral region by the spectral measure at the root of an infinite regular tree, whose root has the same degree as $x$ and all its children have degree $d$. In this paper, we establish this approximation in the full region where these spectral measures are regular. The main observation underlying our proof is that this spectral measure develops a singularity near the origin if and only if the normalized degree of $x$ (see \eqref{eq:def_alpha} below) is small, and it also develops a singularity in the interval $\pm [2,\infty)$ if and only if the normalized degree of $x$ is at least $2$. We combine this observation with elementary tail bounds on the maximal and minimal degrees of $\bb G(N,d/N)$.

The rest of this paper is devoted to the proof of Theorem \ref{thm:delocalization_ER}.
\paragraph{Notation} 
Every quantity that is not explicitly \emph{constant} depends on $N$. In statements of conditions we use $\kappa$ to denote a positive constant. We use $X = O(Y)$ and $\abs{X} \lesssim Y$ interchangeably to mean $\abs{X} \leq C Y$ for some constant $C$. We write $X \asymp Y$ if $X \lesssim Y$ and $X \gtrsim Y$. 
We abbreviate $[N] \deq \{1, 2, \dots, N\}$.

\section{Conditional delocalization for sparse matrices}

In this section, we state a version of Theorem~\ref{thm:delocalization_ER} in the more general setup 
of sparse matrices in Proposition~\ref{pro:delocalization_general} below. 
In Subsection~\ref{subsec:proof_main_result}, we then conclude Theorem~\ref{thm:delocalization_ER} 
from Proposition~\ref{pro:delocalization_general} by analysing the degree distribution of Erd{\H o}s-R\'enyi graphs.

We now introduce these sparse matrices, which generalize (appropriately scaled) adjacency matrices of Erd{\H o}s-R\'enyi graphs. 
We consider matrices $M$ of the form 
\begin{equation} \label{eq:def_M} 
M = H + f \f e \f e^*\,.
\end{equation}
Here, $0 \leq f \leq N^{\kappa/6}$, $\f e \deq N^{-1/2}(1,1,\dots,1)^*$ and $H= (H_{ij}) \in \C^{N\times N}$ is a Hermitian random matrix satisfying the following assumptions 
for some $d$ with 
\begin{equation} \label{eq:d_regime} 
\kappa \log N \leq d \leq \kappa^{-1} \log N. 
\end{equation} 
\begin{enumerate}[label = (A\arabic*)] 
	\item \label{item:assum_A1} The upper-triangular entries ($H_{ij}\col 1 \leq i \leq j\leq N$) are independent.
	\item \label{item:assum_A2} We have $\bb E H_{ij}=0$ and $ \bb E \abs{H_{ij}}^2=(1 + O(\delta_{ij}))/N$ for all $i,j$.
	\item \label{item:assum_A3} Almost surely, $\abs{H_{ij}} \leq \kappa^{-1} d^{-1/2}$ for all $i,j$.
\end{enumerate}

\noindent We assume throughout the remainder of this paper that \eqref{eq:d_regime} holds. 

The next proposition is our main result on the eigenvector delocalization of matrices $M$ 
as defined in \eqref{eq:def_M}. For its formulation, we need the following notion of high probability events.
We say that a (possibly $N$-dependent) event $\Xi$ occurs \emph{with very high probability} if 
for each constant $\nu>0$ there is a constant $C>0$ such that $\P(\Xi) \geq 1 - C N^{-\nu}$ for all sufficiently large $N$. 
Moreover, we say that an event $\Xi$ occurs \emph{with very high probability on an event $\Omega$} 
if for each constant $\nu>0$ there is a constant $C>0$ such that $\P(\Xi\cap \Omega) \geq \P(\Omega) - 
CN^{-\nu}$ for all sufficiently large $N$. 
The eigenvector delocalization of $M$ turns out to depend on the behaviour of the $\ell^2$-norms of 
the columns of $H$ which we denote by 
\begin{equation} \label{eq:def_beta_x} 
 \beta_x \deq \sum_{y \in [N]} \abs{H_{xy}}^2. 
\end{equation}

\begin{proposition} \label{pro:delocalization_general} 
Let $\kappa >0$ be a constant. 
Let $M$ be as in \eqref{eq:def_M} such that \ref{item:assum_A1} -- \ref{item:assum_A3} and \eqref{eq:d_regime} hold. 

\begin{enumerate}[label=(\roman*)] 
\item \label{item:deloc_lower_bound}
With very high probability on the event $\{ \beta_x \geq \kappa \text{ for all } x \}$, all eigenvectors $\f u$ of $M$ with eigenvalues in $[-2+ \kappa, 2 - \kappa]$ satisfy $q(\f u) \leq N^{-1 + \kappa}$.
\item \label{item:deloc_upper_bound} With very high probability on the event $\{ \beta_x \leq 2 - \kappa \text{ for all } x \}$, all eigenvectors $\f u$ of $M$ with eigenvalues 
outside $(-\kappa,\kappa)$ satisfy $q(\f u) \leq N^{-1 + \kappa}$.
\end{enumerate} 
\end{proposition} 

\noindent The proof of Proposition~\ref{pro:delocalization_general} shall be given directly after the statement of Lemma~\ref{lem:m_alpha} below.

\subsection{Delocalization for Erd{\H o}s-R\'enyi graphs -- proof of Theorem~\ref{thm:delocalization_ER}} \label{subsec:proof_main_result} 

Let now $A$ be the adjacency matrix of the Erd{\H o}s-R\'enyi graph $\mathbb{G}(N,d/N)$. 
For any vertex $x \in [N]$, we define its \emph{normalized degree}
\begin{equation} \label{eq:def_alpha}
\alpha_x \deq \frac{1}{d} \sum_{y \in [N]} A_{xy}.
\end{equation}

The next lemma provides tail bounds for the extreme normalized degrees.

\begin{lemma} \label{pro:degree_ER} 
Let $0 \leq d \leq \sqrt{N}$ and write $d = b \log N$.
For any small enough $\epsilon > 0$ we have the bounds
\begin{align} \label{degree_est_1}
\P(\exists x \col \alpha_x \geq 2 - \epsilon) &\leq 2 \exp \qbb{- \log N \pbb{\frac{b}{b_*} - 1 - 2 \epsilon}}\,,
\\ \label{degree_est_2}
\P(\exists x \col \alpha_x \leq \epsilon) &\leq 2 \exp \qB{- \log N \pb{b (1 + 2 \epsilon \log \epsilon) - 1}}\,.
\end{align}
\end{lemma} 

We defer the proof of Lemma~\ref{pro:degree_ER} to the end of this section and 
first combine it with Proposition~\ref{pro:delocalization_general} to deduce Theorem~\ref{thm:delocalization_ER}. 
Proposition~\ref{pro:delocalization_general} is applicable as 
 $M = d^{-1/2} A$ satisfies \eqref{eq:def_M} with $f = \sqrt{d} \lesssim (\log N)^{1/2} \leq N^{\kappa/6}$. 
Moreover, we have 
\begin{equation} \label{eq:relation_beta_alpha} 
 \beta_x = \alpha_x + O\bigg(\frac{d + \log N}{N} \bigg) 
\end{equation} 
with very high probability (see \cite[Remark~4.3]{ADK20} for details). 

\begin{proof}[Proof of Theorem~\ref{thm:delocalization_ER}] 
For $d \geq C \log N$, Theorem~\ref{thm:delocalization_ER} has been proved in \cite{HeKnowlesMarcozzi2018}. 
Hence, we restrict to the regime \eqref{eq:d_regime} in the remainder of this proof. 

We start with the proof of \ref{item:deloc_throughout_bulk}. From \eqref{degree_est_2}, \eqref{eq:relation_beta_alpha} and the lower bound on $d$ in \ref{item:deloc_throughout_bulk},
 we conclude that the event $\{ \beta_x \geq \kappa \text{ for all }x \}$ occurs with probability $1 - o(1)$. Therefore, \ref{item:deloc_throughout_bulk} follows from Proposition~\ref{pro:delocalization_general} 
\ref{item:deloc_lower_bound} with the choice $M = d^{-1/2} A$. 

If the lower bound on $d$ from \ref{item:deloc_everywhere} holds then 
\eqref{degree_est_1}, \eqref{degree_est_2} and \eqref{eq:relation_beta_alpha} yield that 
$\{ \kappa \leq \beta_x \leq 2- \kappa \text{ for all } x\}$ occurs with probability $1 - o(1)$. 
Hence, \ref{item:deloc_everywhere} is a consequence of 
Proposition~\ref{pro:delocalization_general}. 
\end{proof}

\begin{proof}[Proof of Lemma~\ref{pro:degree_ER}] 
The estimate \eqref{degree_est_1} is an easy consequence of Bennett's inequality (see e.g. \cite[Section 2.7]{BLM13})
and a union bound. 
For the proof of \eqref{degree_est_2}, by a union bound and standard Poisson approximation results (see e.g.\ \cite[Lemmas D.1, D.3, D.4]{ADK21}), it suffices to prove that, if $X$ denotes a Poisson random variable with expectation $d$, then $N \P(X \leq \epsilon d)$ is bounded by the right-hand side of \eqref{degree_est_2} multiplied by $3/4$. 
Indeed, for $\epsilon$ small enough, we have 
\begin{equation*}
\P(X \leq \epsilon d) = \sum_{k  \leq \floor{\epsilon d}} \frac{d^k}{k!} \, \dd^{-d} \leq \frac{4}{3} \frac{d^{\floor{\epsilon d}}}{\floor{\epsilon d}!} \ee^{-d} \leq \frac{3}{2} \ee^{-d + \epsilon d (1 - \log \epsilon)} \leq \frac{3}{2} \exp \qB{- b (1 + 2 \epsilon \log \epsilon) \log N }\,,
\end{equation*}
where in the third step we used Stirling's approximation.
\end{proof}

\section{Conditional local law for sparse matrices and proof of Proposition~\ref{pro:delocalization_general}}

We now introduce the notation required for the main result of this section, Theorem~\ref{thm:local_law} below. 
We start with the indicator functions 
\[ \psi_\mathrm{l} \deq \ind{\beta_x \geq \kappa \text{ for all }x}, \qquad 
\psi_\mathrm{u} \deq \ind{\beta_x \leq 2 - \kappa \text{ for all }x}, 
\] 
which impose lower and upper bounds on all $\beta_x$, respectively. 
Moreover, we define the associated spectral domains 
\[ 
\mathbf{S}_\mathrm{l} \deq [-2 + \kappa, 2-\kappa] \times [N^{-1+ \kappa}, 1],  
\qquad \mathbf{S}_\mathrm{u} \deq (\R \setminus (-\kappa, \kappa))\times [N^{-1+ \kappa}, 1].  \] 
For $z \in \C$ with $\Im z >0$ and $\alpha \geq 0$, we introduce 
\begin{equation} \label{eq:def_m_alpha} 
 m_\alpha(z) \deq - \frac{1}{z + \alpha m(z)}, \qquad m(z) \deq \frac{-z + \sqrt{z^2 - 4}}{2},
\end{equation}
where the square root is chosen so that $m$ has a branch cut on $[-2,2]$.
The function $m$ is the Stieltjes transform of the semicircle law, and $m_\alpha$ is 
the Stieltjes transform of an explicit probability measure on $\R$ (see \cite[eq.~(4.11), (4.12) and the 
surrounding explanations]{ADK20}).  

Finally, we denote the resolvent of $M$ by $G(z) \deq (M-z)^{-1}$ with entries $G_{xy}(z)$. 

\begin{theorem}[Conditional local law] \label{thm:local_law} 
Let $M$ be as in \eqref{eq:def_M} such that \ref{item:assum_A1} -- \ref{item:assum_A3} and \eqref{eq:d_regime} hold. 
Let $\# \in \{ \mathrm{l}, \mathrm{u} \}$. 
Then we have 
\[ \psi_\# \max_{x,y \in [N]} \abs{G_{xy}(z) - \delta_{xy} m_{\beta_x}(z)} \lesssim (\log N)^{-1/7}, 
\qquad \psi_\# \absbb{\frac{1}{N} \mathrm{Tr} \, G(z) - m(z)} \lesssim (\log N)^{-1/7} \] 
for all $z \in \mathbf{S}_\#$ with very high probability. 
\end{theorem} 

The proof of Theorem~\ref{thm:local_law} is given at the end of Subsection~\ref{sec:proof_local_law}. 
The next lemma collects a few basic properties of $m_\alpha$. 

\begin{lemma}[Properties of $m_\alpha$]  \label{lem:m_alpha} 
Let $\kappa >0$ be a constant. 
If $\alpha \geq \kappa$ and $z \in \mathbf{S}_{\mathrm{l}}$, or $\alpha \leq 2-\kappa$ and $z \in \mathbf{S}_{\mathrm{u}}$, then
\begin{align} 
\abs{m_{\alpha}(z)} & \lesssim 1, \label{eq:upper_bound_m_alpha} \\ 
\abs{m_{\alpha}(z) - m(z)} & \lesssim \abs{\alpha - 1}. \label{eq:continuity_m_alpha_in_alpha}
\end{align} 
\end{lemma} 

We defer the proof of Lemma~\ref{lem:m_alpha} to the end of this subsection and 
first combine it with Theorem~\ref{thm:local_law} to prove Proposition~\ref{pro:delocalization_general}.

\begin{proof}[Proof of Proposition~\ref{pro:delocalization_general}] 
Following \cite[proof of Theorem~1.8]{ADK20} as well as replacing \cite[Theorem~4.2]{ADK20} by 
 Theorem~\ref{thm:local_law} and \cite[eq.~(A.4)]{ADK20} by \eqref{eq:upper_bound_m_alpha} 
yield Proposition~\ref{pro:delocalization_general}.  
\end{proof}

For the following proof, we note that $m(z)$ is the unique solution of the self-consistent equation
\begin{equation} \label{eq:self_consistent_m} 
m(z) = - \frac{1}{z + m(z)}
\end{equation}
with $\Im m(z) >0$ for $\Im z>0$.

\begin{proof}[Proof of Lemma~\ref{lem:m_alpha}]  
The definition of $m_\alpha$ in \eqref{eq:def_m_alpha} together with 
\eqref{eq:self_consistent_m} implies
\[ m_\alpha - m = m^2 m_\alpha ( \alpha - 1). \] 
Therefore, to show Lemma~\ref{lem:m_alpha}, it suffices to prove the boundedness of $m_\alpha$, 
i.e.\ \eqref{eq:upper_bound_m_alpha}, in the respective domains. 
We will use some standard properties of $m$ in the following. Their proofs can e.g.\ be found in 
\cite[Lemma~3.3]{BenyachKnowles2017}. 

If $\alpha \geq \kappa$ then the boundedness of $m_\alpha(z)$ for $z \in \mathbf{S}_{\mathrm{l}}$ follows from \eqref{eq:def_m_alpha} and  \eqref{eq:im_m_bounded_from_below}.
We now assume $\alpha \leq 2 - \kappa$ and $z \in \mathbf{S}_{\mathrm{u}}$. 
If $\alpha \leq 2 - \kappa$ and $\abs{\Re z}\geq 2 - \kappa/2$ then 
$\abs{z + \alpha m} \geq \kappa/2$ by \eqref{eq:m_leq_1} 
and, hence, the boundedness follows from \eqref{eq:def_m_alpha}. 
If $\alpha \geq \kappa/2$ and $\kappa \leq \abs{\Re z} \leq 2 - \kappa/2$ then the boundedness has been established before. 
Finally, if $\alpha \leq \kappa/2$ then $\abs{ z + \alpha m} \geq \kappa/2$ 
due to \eqref{eq:m_leq_1} and $\abs{z} \geq \kappa$ for $z \in \mathbf{S}_{\mathrm{u}}$.
\end{proof}

\subsection{Proof of local law -- Theorem~\ref{thm:local_law}} \label{sec:proof_local_law} 

For the proof of Theorem~\ref{thm:local_law}, we call a vertex $x \in [N]$ \emph{typical} \cite{ADK20} if 
\begin{equation} \label{eq:def_typical} 
 \absbb{\sum_{y \neq x} \bigg( \abs{H_{xy}}^2 - \frac{1}{N} \bigg)} \leq (\log N)^{-1/3} 
\quad \text{ and } \qquad \absbb{\sum_{y \neq x} \bigg( \abs{H_{xy}}^2 - \frac{1}{N} \bigg)G_{yy}^{(x)} } \leq (\log N)^{-1/3}, 
\end{equation} 
where $G_{yy}^{(x)} \deq \big( ( (M_{ab})_{a,b \in [N] \setminus \{x \}} - z )^{-1}\big)_{yy}$. 
We denote the set of typical vertices by $\cal T$. 
Note that $\cal T$ depends on the spectral parameter $z$.

Furthermore, we introduce the control parameter $\Lambda$, the indicator function $\phi_t$ with $t > 0$ 
defined through 
\begin{equation} \label{eq:def_Lambda_phi_t} 
\Lambda(z) \deq \max_{x,y \in [N]} \abs{G_{xy}(z) - \delta_{xy} m_{\beta_x}(z)}, \qquad \quad \phi_t(z) \deq \mathbbm{1}_{\Lambda(z) \leq (\log N)^{-1/t}}, 
\end{equation}  
and the average of the diagonal resolvent entries on the typical vertices 
\begin{equation} \label{eq:def_s} 
s(z) \deq \abs{\cal T}^{-1} \sum_{x \in \cal T} G_{xx}(z). 
\end{equation} 

\begin{lemma}[Approximate self-consistent equation] \label{lem:self_consistent} 
Let $\# \in \{ \mathrm{l}, \mathrm{u} \}$. 
Then for each constant $t >0$, we have 
\[ 
(\psi_\# \phi_t (z) + \mathbbm{1}_{\Im z = 1})  (1 + z G_{xx}(z) + s(z)G_{xx}(z) ) = O( (\log N)^{-1/3})
\] 
for $x \in \cal T$ and $z \in \mathbf{S}_{\#}$
with very high probability. 
\end{lemma} 

\begin{proof} 
This follows directly from some results of \cite{ADK20}. 
We recall the definition $\theta\deq\ind{\max_{x,y} \abs{G_{xy}} \leq \Gamma}$ from \cite[eq.~(4.19)]{ADK20}. 
Note that $\psi_\# \phi_t \leq \theta$ on $\mathbf{S}_{\#}$ by \eqref{eq:upper_bound_m_alpha} if $\Gamma$ is chosen large enough. 
Moreover, $\ind{\Im z = 1} \leq 
\theta$ as $\abs{G_{xy}} \leq (\Im z)^{-1}$ trivially and $\Gamma \geq 1$. 
By \cite[Proposition~4.8~(i)]{ADK20} with $\varphi_\mathfrak{a} \asymp (\log N)^{-1/3}$ and 
$d \asymp \log N$, we deduce the estimate
\begin{equation} \label{eq:Tc_estimate}
\theta \abs{\cal T^c} \lesssim N \exp(-c (\log N)^{1/3}) \lesssim N (\log N)^{-1/3}
\end{equation}
with very high probability.
Using these observations, from \cite[Lemma~4.16]{ADK20}, \cite[eq.~(4.38a)]{ADK20}, the definition 
of a typical vertex in \eqref{eq:def_typical}, \cite[eq.~(4.38c)]{ADK20}, and \eqref{eq:Tc_estimate},
we obtain Lemma~\ref{lem:self_consistent} 
(see also \cite[eq.~(4.44)]{ADK20} with $d \asymp \log N$ by \eqref{eq:d_regime} and $\varphi_{\mathfrak a} \asymp (\log N)^{-1/3}$).  
\end{proof} 

Averaging over $x \in \cal T$ in Lemma~\ref{lem:self_consistent} shows 
that $s(z)$ satisfies an approximate version of the self-consistent equation \eqref{eq:self_consistent_m} for $m(z)$ defined in \eqref{eq:def_m_alpha}. The self-consistent equation \eqref{eq:self_consistent_m} has another solution, denoted by $\wt m(z)$, whose imaginary part is negative. It is given by
\begin{equation}
\label{eq:def_m_tilde} 
\wt{m}(z) \deq \frac{-z - \sqrt{z^2 - 4}}{2}.
\end{equation}

\begin{lemma}[Initial bound and bootstrapping steps]  \label{lem:bootstrapping} 
Let $\# \in \{ \mathrm{l}, \mathrm{u} \}$ and $z \in \mathbf{S}_{\#}$. Then the following holds. 
\begin{enumerate}[label=(\roman*)] 
\item \label{item:bound_Lambda_Im_z_equals_1} 
If $\Im z = 1$ then, with very high probability, 
\[ \Lambda(z) \lesssim (\log N)^{-1/6}. \] 
\item \label{item:m_minus_wt_m_large} If $\abs{m(z) - \wt{m}(z)} > 2 (\log N)^{-1/7}$ then, with very high probability, 
\[ \psi_\# \phi_7(z)  \Lambda(z) \lesssim (\log N)^{-1/6}. \]  
\item  \label{item:m_minus_wt_m_small} If $\abs{m(z) - \wt{m}(z)} \leq 2 (\log N)^{-1/7}$ then, with very high probability, 
\[ \psi_\# \phi_8(z) \Lambda(z) \lesssim (\log N)^{-1/7}.   
\] 
\end{enumerate}
\end{lemma} 

\begin{proof} 
We first note that throughout the following arguments it suffices to consider the diagonal terms, 
$G_{xx} - m_{\beta_x}$,  
in the definition of $\Lambda$, since $\theta \max_{x\neq y} \abs{G_{xy}} \lesssim (\log N)^{-1/2}$ 
with very high probability by \cite[eq.~(4.38b)]{ADK20}. 
We refer to the proof of Lemma~\ref{lem:self_consistent} for the definition of $\theta$ and 
the proof that $\theta = 1$ in all cases considered in Lemma~\ref{lem:bootstrapping}. 

Let $t >0$ be a constant. 
By averaging over $x \in\cal T$ in Lemma~\ref{lem:self_consistent}, we conclude that 
\[ (\psi_\# \phi_t + \ind{\Im z =1}) (1 + z s + s^2) = O\big((\log N)^{-1/3}\big) \] 
for $z \in \mathbf{S}_\#$ with very high probability. 
Therefore, standard stability estimates for \eqref{eq:self_consistent_m} (e.g.\ \cite[Lemma~4.4]{HeKnowlesMarcozzi2018}) yield 
\begin{equation} \label{eq:stability_conclusion} 
 (\psi_\# \phi_t  + \ind{\Im z =1})\min \{ \abs{s-m},\abs{s-\wt{m}} \} \lesssim (\log N)^{-1/6} 
\end{equation}
for $z \in \mathbf{S}_{\#}$ with very high probability. 

If $\Im z = 1$ then we conclude from \eqref{eq:stability_conclusion} that $\abs{m - s} \lesssim (\log N)^{-1/6}$ 
since $\Im s >0$, $\Im m>0$, $\Im \wt{m} <0$ and $\abs{m - \wt{m}} \gtrsim 1$. 
Together with Lemma~\ref{lem:self_consistent}, for $x \in \cal T$, this implies 
\[ 
G_{xx} =  - \frac{ 1 + O( (\log N)^{-1/3})}{z + s } = m  + O ( (\log N)^{-1/6})  
\] 
if $\Im z =1$ due to \eqref{eq:self_consistent_m} and \eqref{eq:m_leq_1}.
Thus, 
$\abs{G_{xx} - m_{\beta_x}} \leq \abs{G_{xx} - m} + \abs{m-m_{\beta_x}} = O((\log N)^{-1/6})$ by \eqref{eq:continuity_m_alpha_in_alpha}. 
This proves 
\begin{equation} \label{eq:aux_bound_typical_vertices} 
\max_{x \in \cal T} \abs{G_{xx} - m_{\beta_x}} \lesssim (\log N)^{-1/6}.  
\end{equation}

Next, let $x \notin \cal T$. 
Following \cite[proof of eq.~(4.48)]{ADK20} and using \eqref{eq:aux_bound_typical_vertices} instead 
of \cite[eq.~(4.46)]{ADK20}, we obtain  
\begin{equation} \label{eq:expansion_atypical_vertex} 
 G_{xx} - m_{\beta_x} = - m_{\beta_x} \frac{1}{-z - \beta_x m + \eps_x} \eps_x 
\end{equation}
with very high probability if $\Im z =1$, where $\eps_x = O((1 + \beta_x) (\log N)^{-1/6})$. 
Since $m_{\beta_x}$ is the Stieltjes transform of a probability measure on $\R$, we have 
$\abs{m_{\beta_x}} \leq (\Im z)^{-1}$ and $\Im m_{\beta_x} \gtrsim 1$ if $\Im z = 1$. 
Hence, \eqref{eq:expansion_atypical_vertex} implies 
$\abs{G_{xx}- m_{\beta_x}} \lesssim (\log N)^{-1/6}$ with very high probability. 
This proves \ref{item:bound_Lambda_Im_z_equals_1}. 

If $\abs{m - \wt{m}} > 2 (\log N)^{-1/7}$ then we conclude from 
\eqref{eq:continuity_m_alpha_in_alpha},  \eqref{eq:def_typical}, \eqref{eq:stability_conclusion} 
and the definitions of $\phi_7$ and $s$ from \eqref{eq:def_Lambda_phi_t} and \eqref{eq:def_s}, respectively,  
that $\psi_\# \phi_7 \abs{s - m} \lesssim (\log N)^{-1/6}$. 
Then we follow the arguments in the proof of \ref{item:bound_Lambda_Im_z_equals_1} 
and obtain that \eqref{eq:aux_bound_typical_vertices} and \eqref{eq:expansion_atypical_vertex} 
hold with very high probability on the event $\{\psi_\# \phi_7 = 1\}$. 

We now estimate the right-hand side of \eqref{eq:expansion_atypical_vertex}. 
From \eqref{eq:im_m_bounded_from_below} and $\phi_7 \psi_\mathrm{l} \eps_x =  
O( (1 + \beta_x) (\log N)^{-1/6})$, we conclude that $\abs{- z - \beta_x m + \eps_x} 
\geq \abs{\Im z + \beta_x \Im m - \Im \eps_x} \gtrsim \beta_x$. 
Thus, using $\abs{\eps_x} \lesssim \beta_x (\log N)^{-1/6}$ by $\beta_x \geq \kappa$ 
and \eqref{eq:upper_bound_m_alpha} in 
\eqref{eq:expansion_atypical_vertex} completes the proof of \ref{item:m_minus_wt_m_large} 
if $\# = \mathrm{l}$. 
If $\# = \mathrm{u}$ then $\beta_x \leq 2 - \kappa$ implies 
 $\abs{\eps_x} \lesssim (\log N)^{-1/6}$. 
Hence, $\abs{-z - \beta_x m + \eps_x} = \abs{m_{\beta_x}^{-1} + \eps_x} \gtrsim 1$ 
due to \eqref{eq:def_m_alpha} and 
$\abs{m_{\beta_x}} \lesssim 1$ on $\mathbf{S}_{\mathrm{u}}$ by \eqref{eq:upper_bound_m_alpha}. 
Applying these estimates to the right-hand side of \eqref{eq:expansion_atypical_vertex} 
 completes the proof of \ref{item:m_minus_wt_m_large}. 

For the proof of \ref{item:m_minus_wt_m_small}, we note that \eqref{eq:stability_conclusion} 
and the condition $\abs{m-\wt{m}} \leq 2 ( \log N)^{-1/7}$ imply 
$\psi_\# \phi_8 \abs{s - m} \lesssim (\log N)^{-1/7}$. 
Hence, proceeding as in the proof of \ref{item:m_minus_wt_m_large} 
shows \ref{item:m_minus_wt_m_small}. 
\end{proof} 

We now establish Theorem~\ref{thm:local_law} by showing that $\phi_7 =1 =\phi_8$ for $\Im z =1$ 
and bootstrapping this information to small values of $\Im z$ using Lemma~\ref{lem:bootstrapping}. 

\begin{proof}[Proof of Theorem~\ref{thm:local_law}] 
Fix $z \in \mathbf{S}_{\#}$. Set $z_k \deq \Re z + \ii \max\{ (1 - N^{-3}k), \Im z\}$ for $k \in \N$ 
and $K \deq \min \{ k \in \N \colon \Im z_k = \Im z \}$. 
Note that $z_k \in \mathbf{S}_\#$ for all $k \in \N$. 
Since $\Im z \mapsto \abs{m(z) - \wt{m}(z)}$ is monotonically increasing (as can be seen by an explicit computation using \eqref{eq:def_m_alpha} and \eqref{eq:def_m_tilde}), there is a unique $K_* \in \N$ 
such that $0 \leq K_* \leq K$, $\abs{m(z_k) - \wt{m}(z_k)} > 2 (\log N)^{-1/7}$ for all $k \leq K_*$ and $\abs{m(z_k) - \wt{m}(z_k)} \leq 2 (\log N)^{-1/7} $ for all $k \in (K_*, K] \cap \N$. 
Note that $K_* >0$ while $K_* = K$ is possible.

Throughout the remainder of the argument, we work on the event $\{ \psi_\# = 1 \}$. 

We now show by induction that $\phi_7(z_k) = 1$ for all $k \leq K_*$ with very high probability. 
By Lemma~\ref{lem:bootstrapping} \ref{item:bound_Lambda_Im_z_equals_1}, we have $\Lambda(z_0) \lesssim (\log N)^{-1/6}$, i.e.\ $\phi_7(z_0) = 1$. 

We assume that $\phi_7(z_{k-1}) = 1$ for some $k \leq K_*$. 
The resolvent entries $G_{xy}$ and $m_{\beta_x}$ are 
 Lipschitz-continuous in $z$ with constant $N^2$ if $\Im z \geq N^{-1}$. 
Hence, owing to $\Im z_k \geq N^{-1}$ for all $k$, we have 
\begin{equation} \label{eq:Lambda_step} 
 \Lambda(z_k) \leq \Lambda(z_{k-1}) + 2 N^{-1}. 
\end{equation}
Since $\Lambda(z_{k-1}) \lesssim (\log N)^{-1/6}$ by Lemma~\ref{lem:bootstrapping} \ref{item:m_minus_wt_m_large}, we conclude from \eqref{eq:Lambda_step} that $\Lambda(z_k) \leq (\log N)^{-1/7}$, i.e.\ $\phi_7(z_k)= 1$. 

Therefore, $\phi_7(z_{K_*}) = 1$ and \eqref{eq:Lambda_step} imply $\Lambda(z_{K_* +1}) \leq (\log N)^{-1/8}$, i.e.\ $\phi_8(z_{K_*+1}) = 1$. 
Arguing as above with Lemma~\ref{lem:bootstrapping} \ref{item:m_minus_wt_m_large}
replaced by Lemma~\ref{lem:bootstrapping} \ref{item:m_minus_wt_m_small}, we obtain 
$\phi_8(z_k) = 1$ for all $k \leq K$. 
In particular, $\psi_\# \Lambda = \psi_\# \Lambda(z_K) \lesssim (\log N)^{-1/7}$ by 
Lemma~\ref{lem:bootstrapping} \ref{item:m_minus_wt_m_small}. 
This proves the first bound in Theorem~\ref{thm:local_law}. 

For the second bound, we note that 
$m_{\beta_x} = m + O( (\log N)^{-1/3})$ for $x \in \cal T$ by \eqref{eq:continuity_m_alpha_in_alpha} and 
\eqref{eq:def_typical} 
while 
$\abs{G_{xx} - m_{\beta_x}} + \abs{m_{\beta_x}} + \abs{m} \lesssim 1$ for $x \notin \cal T$ by \eqref{eq:m_leq_1}. 
Therefore, the second bound in Theorem~\ref{thm:local_law} follows 
from averaging the first bound in Theorem~\ref{thm:local_law} over $x \in [N]$, 
distinguishing the cases $x \in \cal T$ and $x \notin \cal T$ 
and using \eqref{eq:Tc_estimate}. 
\end{proof}

In the case $\# = \mathrm{l}$, Theorem~\ref{thm:local_law} can be proved by adjusting some arguments 
from \cite{ADK20} as we explain in the next remark. 

\begin{remark}[Alternative proof of Theorem~\ref{thm:local_law} for $\# = \mathrm{l}$] 
If $\# = \mathrm{l}$ then following \cite[proof of Proposition~4.18]{ADK20}
and replacing \cite[eq.~(A.4) and (A.5)]{ADK20} by \eqref{eq:upper_bound_m_alpha} and 
\eqref{eq:continuity_m_alpha_in_alpha}, respectively, yields 
\begin{equation} \label{eq:bootstrapping_o_alternative} 
 \mathbbm{1}_{\Lambda \leq \lambda} \psi_{\mathrm{l}} \Lambda \leq \mathcal C (\log N)^{-1/3}
\end{equation}
for all $z \in \mathbf{S}_{\mathrm{l}}$ with very high probability, where $\lambda \leq 1$ is an arbitrary constant. 
Then we obtain Theorem~\ref{thm:local_law} by following \cite[proof of Theorem 4.2]{ADK20}, 
choosing $\lambda$ appropriately
and using \eqref{eq:bootstrapping_o_alternative} instead of \cite[Proposition~4.18]{ADK20}. 
\end{remark}

\appendix 

\section{Basic estimates on $m$} 

The next lemma collects a few basic estimates of the Stieltjes transform $m$ defined in \eqref{eq:def_m_alpha}. 

\begin{lemma}
For each $z \in \C$ with $\Im z>0$, we have 
\begin{equation} \label{eq:m_leq_1} 
\abs{m(z)} \leq 1.  
\end{equation} 
For each $z \in \mathbf{S}_{\mathrm{l}}$, we have  
\begin{equation} \label{eq:im_m_bounded_from_below} 
\Im m(z) \gtrsim 1.  
\end{equation} 
\end{lemma}

\begin{proof} 
The upper bound \eqref{eq:m_leq_1} follows from inverting \eqref{eq:self_consistent_m}, 
taking the imaginary part of the result and using that $\Im m(z)>0$. 
For the proof of \eqref{eq:im_m_bounded_from_below}, we refer to \cite[eq.~(3.3)]{BenyachKnowles2017}. 
\end{proof}

\medskip

\paragraph{Acknowledgements}
The authors acknowledge funding from the European Research Council (ERC) under the European Union’s Horizon 2020 research and innovation programme, grant agreement No.\ 715539\_RandMat and the Marie Sklodowska-Curie grant agreement No.\ 895698.  Funding from the Swiss National Science Foundation through the NCCR SwissMAP grant is also acknowledged. 
J.A.\ and A.K.\ acknowledge support from the National Science Foundation under Grant No.\ DMS-1928930 during their participation in the program ``Universality and Integrability in Random Matrix Theory and Interacting Particle Systems'' hosted by the Mathematical Sciences Research Institute in Berkeley, California during the Fall semester of 2021.

\providecommand{\bysame}{\leavevmode\hbox to3em{\hrulefill}\thinspace}
\providecommand{\MR}{\relax\ifhmode\unskip\space\fi MR }
\providecommand{\MRhref}[2]{%
  \href{http://www.ams.org/mathscinet-getitem?mr=#1}{#2}
}
\providecommand{\href}[2]{#2}

\noindent
Johannes Alt (\href{mailto:johannes.alt@unige.ch}{johannes.alt@unige.ch}) -- University of Geneva and New York University.
\\
Rapha\"el Ducatez (\href{mailto:raphael.ducatez@ens-lyon.fr}{raphael.ducatez@ens-lyon.fr}) -- ENS Lyon, Unité de Mathématiques Pures et Appliqués (UMPA).
\\
Antti Knowles (\href{mailto:antti.knowles@unige.ch}{antti.knowles@unige.ch}) -- University of Geneva.


\begin{thebibliography}{10}

\bibitem{ADK20}
J.~Alt, R.~Ducatez, and A.~Knowles, \emph{Delocalization transition for
  critical {E}rd{\H{o}}s--{R}\'{e}nyi graphs}, Comm. Math. Phys. \textbf{388}
  (2021), no.~1, 507--579.

\bibitem{ADK19}
J.~Alt, R.~Ducatez, and A.~Knowles, \emph{Extremal eigenvalues of critical
  {E}rd{\H{o}}s--{R}{\'e}nyi graphs}, Ann. Prob. \textbf{49} (2021), no.~3,
  1347--1401.

\bibitem{ADK21}
J.~Alt, R.~Ducatez, and A.~Knowles, \emph{Poisson statistics and localization
  at the spectral edge of sparse {E}rd{\H o}s--{R}\'enyi graphs}, Preprint
  arXiv:2106.12519 (2021).

\bibitem{BBK2}
F.~Benaych-Georges, C.~Bordenave, and A.~Knowles, \emph{Largest eigenvalues of
  sparse inhomogeneous {E}rd{\H o}s-{R}\'enyi graphs}, Ann. Prob. \textbf{47}
  (2019), no.~3, 1653--1676.

\bibitem{BBK1}
F.~Benaych-Georges, C.~Bordenave, and A.~Knowles, \emph{Spectral radii of
  sparse random matrices}, Ann. Inst. Henri Poincar{\'e} \textbf{56} (2020),
  no.~3, 2141--2161.

\bibitem{BenyachKnowles2017}
F.~Benaych-Georges and A.~Knowles, \emph{Local semicircle law for {W}igner
  matrices}, Advanced topics in random matrices, Panor. Synth\`eses, vol.~53,
  Soc. Math. France, Paris, 2017, pp.~1--90.

\bibitem{Bol01}
B.~Bollob\'as, \emph{Random graphs}, Cambridge University Press, 2001.

\bibitem{BLM13}
S.~Boucheron, G.~Lugosi, and P.~Massart, \emph{Concentration inequalities: A
  nonasymptotic theory of independence}, Oxford university press, 2013.

\bibitem{MR125031}
P.~Erd\H{o}s and A.~R\'{e}nyi, \emph{On the evolution of random graphs}, Magyar
  Tud. Akad. Mat. Kutat\'{o} Int. K\"{o}zl. \textbf{5} (1960), 17--61.

\bibitem{EKYY1}
L.~Erd{\H{o}}s, A.~Knowles, H.-T. Yau, and J.~Yin, \emph{Spectral statistics of
  {E}rd{\H{o}}s-{R}\'enyi graphs {I}: Local semicircle law}, Ann. Prob.
  \textbf{41} (2013), 2279--2375.

\bibitem{ESY2}
L.~Erd{\H{o}}s, B.~Schlein, and H.-T. Yau, \emph{Local semicircle law and
  complete delocalization for {W}igner random matrices}, Comm. Math. Phys.
  \textbf{287} (2009), 641--655.

\bibitem{HeKnowlesMarcozzi2018}
Y.~He, A.~Knowles, and M.~Marcozzi, \emph{Local law and complete eigenvector
  delocalization for supercritical {E}rd{\H o}s-{R}\'{e}nyi graphs}, Ann. Prob.
  \textbf{47} (2019), no.~5, 3278--3302.

\bibitem{tikhomirov2021outliers}
K.~Tikhomirov and P.~Youssef, \emph{Outliers in spectrum of sparse {W}igner
  matrices}, Rand. Struct. Algor. \textbf{58} (2021), no.~3, 517--605.

\end{thebibliography}
\end{document}